\documentclass[a4paper,12pt]{amsart}
\usepackage{amssymb,amsmath}

\usepackage{url}
\usepackage{pgfplots}

\newtheorem{theorem}{Theorem}[section]
\newtheorem{lemma}[theorem]{Lemma}

\newtheorem*{theorem*}{Theorem}

\theoremstyle{definition}
\newtheorem{remark}{Remark}

\newtheorem*{remark*}{Remark}

\newcommand{\euO}{\mathfrak O}
\newcommand{\euP}{\mathfrak P}
\newcommand{\euD}{\mathfrak D}
\newcommand{\euA}{\mathfrak A}
\newcommand{\eub}{\mathfrak b}

\newcommand{\eua}{\mathfrak a}

\newcommand{\euM}{\mathfrak M}

\newcommand{\euT}{\mathfrak T}

\newcommand{\caD}{\mathcal D}

\newcommand{\caH}{\mathcal H}

\newcommand{\bQ}{\mathbb Q}
\newcommand{\bZ}{\mathbb Z}

\newcommand{\bX}{\mathbb X}
\newcommand{\bF}{\mathbb F}

\newcommand{\Tr}{\mathrm{Tr}} 
\newcommand{\Gal}{\mathrm{Gal}}
\newcommand{\chr}{\mathrm{char}}

\begin{document}

\title{Ramified extensions of degree $p$ and their
  Hopf-Galois module structure}

\author{G.~Griffith Elder}

\address{Department of Mathematics \\ University of Nebraska at Omaha\\ Omaha, NE 68182-0243 U.S.A.}  
\email{elder@unomaha.edu}

\subjclass[2010]{Primary 11S15, Secondary 11R33, 16T05}
\keywords{Artin-Schreier equation, Galois module structure}

\begin{abstract}
Cyclic, ramified extensions $L/K$ of degree $p$ of local fields with
residue characteristic $p$ are fairly well understood.  Unless
$\chr(K)=0$ and $L=K(\sqrt[p]{\pi_K})$ for some prime element
$\pi_K\in K$, they are defined by an Artin-Schreier
equation. Additionally, through the work of Ferton, Aiba, de Smit and
Thomas, and others, much is known about their Galois module structure of ideals,
the structure of each ideal $\euP_L^n$ as a module over its associated
order $\euA_{K[G]}(n)=\{x\in K[G]:x\euP_L^n\subseteq \euP_L^n\}$ where
$G=\Gal(L/K)$. This paper extends these results to separable, ramified
extensions of degree $p$ that are not Galois.
\end{abstract}

\date{\today}

\maketitle

\section{Introduction}\label{intro}

Let $K$ be a complete local field with valuation $v_K$ normalized so
that $v_K(K^\times)=\bZ$ and residue field $\kappa$ finite of
characteristic $p>0$. This means that either $K$ is a finite extension
of the $p$-adic numbers $\bQ_p$, or $K$ is the field of Laurent series
$\kappa((X))$ with $X$ indeterminate.  We are interested in ramified
extensions $L$ of degree $p$ over $K$. Certainly, some of these
extensions are generated by a root of a prime element $\pi_K\in K$,
namely, $L=K(x)$ with $x^p=\pi_K$.  Such extensions are special. In
$\chr(K)=p$, these are the inseparable extensions.  We call them {\em
  atypical}, and restrict our attention to {\em typical} extensions,
those that {\em cannot} be generated by a root of a prime element.
For these extensions, we are interested in addressing two classical
questions.

The first
question concerns the defining polynomial.  As is well-known,
when a typical extension $L/K$ is Galois, it
can be defined by a Artin-Schreier polynomial
\begin{equation}\label{AS-equation}
p(x)=x^p-x-\beta \in K[x],
\end{equation}
with its ramification number $b$, defined as in \cite{serre:local},
satisfying $p\nmid b$ and $v_K(\beta)=-b$ \cite{mackenzie:whaples}.
By adjusting the argument of \cite{mackenzie:whaples}, as presented in
\cite[Chapter III \S2]{fesenko:vostokov}, we prove that every typical
extension can be defined by a polynomial of the form
\begin{equation}\label{Sul-equation}
p(x)=x^p-\alpha x-\beta \in K[x],
\end{equation}
where again, the ramification number for $L/K$ is linked in a
transparent manner to the valuations of the coefficients.  Elsewhere, namely
\cite{amano}, such extensions are defined in terms of Eisenstein
polynomials. The value of defining extensions by \eqref{Sul-equation}
is that in additional to a transparent description of ramification,
other properties can be easily described. Indeed, this is why they
were first of interest for global function fields, where they are used
to determine the Hasse-Witt invariant \cite{sullivan,martha:2012}.

The second question concerns Galois module structure, or rather, for a
general typical extension, what must be called Hopf-Galois module
structure.  We begin by introducing the question in its classical
setting, when $L/K$ is Galois with $G=\Gal(L/K)$. Here the search is
for an integral version of the Normal Basis Theorem. Based upon
results of Noether and Leopoldt, the question asks for conditions under which
the ring of integers $\euO_L=\{x\in L:v_L(x)\geq 0\}$ in $L$ is free
over its associated order $\euA_{K[G]}=\{y\in K[G]:y\euO_L\subseteq
\euO_L\}$ in $K[G]$, the largest $\euO_K$-order in $K[G]$ that acts on
$\euO_L$. For general extensions, this and variations of this question
present very difficult problems, and progress starting in the 1970s has been slow. On the other hand, for one specific class of extensions, cyclic of degree
$p$, progress has been good \cite{ferton, aiba,de-smit:thomas,marklove,
  huynh}. One explanation for this progress is that cyclic ramified
extensions of degree $p$ naturally possess a {\em scaffold}.
This is discussed in \cite[\S4.1]{byott:A-scaffold}, although the definition
of scaffold, as presented in \cite{byott:A-scaffold} in its full
generality, may be a challenge to digest. For
extensions of degree $p$ however, a very simple sufficiency condition is
available: If there is an element $x\in L$ with $p\nmid v_L(x)$ and an
element $\Psi\in K[G]$ that ``acts like'' the derivative $d/dx$ on the
$K$-basis $\{x^i\}_{i=0}^{p-1}$ for $L$ over $K$, there is a
scaffold. As we shall see, ``acts like'' is exact in $\chr(K)=p$,
namely $\Psi\cdot x^i=ix^{i-1}$ for $0\leq i<p$. In $\chr(K)=0$, it is
approximate:
$$\Psi\cdot x^i\equiv ix^{i-1}\bmod x^{i-1}\euP_L^\euT,$$ where
$\euP_L=\{x\in L:v_L(x)>0\}$ is the prime ideal and the degree of
approximation is captured by the integer $\euT\geq 1$, 
the {\em tolerance} of the scaffold.

As explained in \cite[\S4.1]{byott:A-scaffold}, using the scaffold and assuming a lower bound on
absolute ramification $v_L(p)>(p-1)(b+2)$, an $\euO_K$-basis for the 
associate order of an ideal $\euP_L^h$,
defined as
$\{y\in K[G]:y\euP_L^h\subseteq \euP_L^h\}$,
can be explicitly described 
\cite[Theorem 3.1]{byott:A-scaffold}. Furthermore, letting
$0\leq \bar{b}<p$ be
the residue of $b\bmod p$, we can conclude that
\begin{enumerate}
\item $\euP_L^{\bar{b}}$ is free over its associated order.
\item $\euO_L$ is free over its associated order if and only if $\bar{b}\mid p-1$.
\item The inverse different $\caD_{L/K}^{-1}$ is free over its associated order if and only if $\bar{b}= p-1$.
\end{enumerate}
These statements also follow from \cite{ferton,
  aiba,de-smit:thomas,marklove, huynh}.  The purpose of this paper is
to extend these cyclic results to typical extensions $L/K$, including
those that are not Galois.  This is accomplished by first using
\cite[\S2]{childs:separable} to identify the unique $K$-Hopf algebra
$\caH$ that acts upon $L$ (making $L$ an $\caH$-Galois extension). This
Hopf algebra has one generator. We explicitly describe the action of
this generator on the $K$-basis $\{x^i\}_{i=0}^{p-1}$ for $L$ with $x$
satisfying \eqref{Sul-equation}, and observe that this yields
scaffold for the $\caH$-action on $L$. Once a
scaffold exists, the main results of \cite{byott:A-scaffold} apply. In
particular, the three statements above hold, with the associate order
of an ideal $\euP_L^h$ in $\caH$ defined as $\{y\in
\caH:y\euP_L^h\subseteq \euP_L^h\}$. Other structural results hold as well. But
for those, we direct the reader to \cite{byott:A-scaffold}.

\begin{remark}
The focus of this paper is on a uniformity of approach, based upon a
certain defining equation. Thus we don't discuss Hopf-Galois module
structure in the setting of atypical extensions. In $\chr(K)=p$, these
extensions are inseparable. See \cite[\S5]{byott:A-scaffold} and
\cite[\S6]{elder:lambda-divided} for two different Hopf algebras that
act upon $L/K$ and a discussion of their resulting Hopf-Galois module
structure.  In $\chr(K)=0$ with $L/K$ Galois, see \cite{ferton}.
\end{remark}

\subsection{Summary of notation}
Let $p$ be a prime.  The field $K$ is either a finite extension of the
$p$-adic numbers (in $\chr(K)=0$), or a field of Laurent series (in
$\chr(K)=p$).  Following common conventions, we use subscripts to
denote field of reference. So $v_K$ is the valuation normalized so
that $v_K(K^\times )=\bZ$, $\pi_K$ is a prime element in $K$ (with
$v_K(\pi_K)=1$), $\euO_K=\{x\in K:v_L(x)\geq 0\}$ is the ring of
integers in $K$. It has a unique maximal ideal $\euP_K=\{x\in
L:v_K(x)\geq 1\}$.  The field $L$ is {\em typical} if it is a ramified
extension of $K$ of degree $p$ that {\em is not} generated by a $p$th
root of a prime element $\pi_K$.
\section{Typical extensions \& ramification}

\begin{theorem}\label{extensions}
If $L/K$ is a typical extension, there are positive integers:
$ef=d\mid p-1$, $0\leq t< e$, $\gcd(t,e)=1$ and $0<b+pt/e<v_L(p)/(p-1)$ with $\gcd(b,p)=1$, as well as two elements:
$\alpha,\beta\in K$ satisfying $v_K(\beta)=-b$, and
$\alpha=\pi_K^{ft}\gamma^f\mu\in \euO_K$ for two units
$\gamma,\mu\in \euO_K^\times$ with $\mu$ representing a coset of order
$f$ in the quotient group $\kappa^\times/(\kappa^\times)^f$ (recall
$\kappa=\euO_K/\euP_K$), such that $L=K(x)$ with
$$x^p-\alpha^{(p-1)/d}x=\beta.$$ Conversely, every such equation yields a typical
extension with
ramification number $$\ell=b+\frac{pt}{e},$$
and different 
$\euD_{L/K}=\euP_L^{(\ell+1)(p-1)}$.
The Galois closure for $L/K$ is $M=K(x,y)$ where $y^d=\alpha$, with
degree of inertia $f$ and ramification index $ep$.
To describe the Galois group,
let $r$ be an integer of order
$d$ modulo $p$. Let $\rho=r$ in $\chr(K)=p$ and let $\rho$ be the
Teichm\"{u}ller character for $r$ (a primitive $d$th root of unity in
the $p$-adic integers $\bZ_p$ such that $\rho\equiv r\bmod p$) when
$\chr(K)=0$. Then $\Gal(M/K)=\langle\sigma,\tau:\sigma^p=\tau^{d}=1,
\tau\sigma\tau^{-1}=\sigma^r\rangle$ with
$\tau(y)=\rho y$, $\tau(x)=x$, $\sigma(y)=y$, and 
$\sigma(x)=x+y+y\Delta$.  In  $\chr(K)=p$, $\Delta=0$. In
$\chr(K)=0$,
 $\Delta\in M$ satisfies
$v_M(\Delta)=v_M(p)-(p-1)e\ell$.
The ramification number of $M/K(y)$ is $e\ell$.
\end{theorem}
\begin{remark}
In $\chr(K)=p$, $v_L(p)=\infty$ and there is no upper bound on
$\ell$.
\end{remark}
\begin{remark}
As mentioned in \S\ref{intro}, our definition of ramification numbers
follows \cite[IV]{serre:local}. For separable non-Galois extensions,
this is done via the upper numbering for the Galois group of the
Galois closure of $L/K$, as suggested by \cite[IV\S3 Remark
  2]{serre:local} and developed more fully in \cite{helou:1991}.  We
need to state this clearly, as there are two alternative ways that ramification numbers can be assigned values.  The simplest way to
explained difference is in terms of the graph of the Hasse-Herbrand
function, which is a piecewise linear function that is graphed either
in $[-1,\infty)\times [-1,\infty)$, as in \cite{serre:local} (see
    graphs on \cite[pages 116, 117]{artin:tate}), or shifted into
    $[0,\infty)\times [0,\infty)$, as in \cite[page 2274]{helou:1991}
        and \cite[\S3]{jones:roberts-paper}.  Ramification numbers
        then correspond with vertices. Lower ramification numbers
        corresponding to first coordinates. Upper ramification numbers
        correspond to second coordinates.  If the Hasse-Herbrand
        function is graphed in $[-1,\infty)\times [-1,\infty)$, we say
            that the ramification numbers (lower or upper) have Serre
            values. If the function is graphed in $[0,\infty)\times
              [0,\infty)$, they have Artin values.  To transition from
                first to the second, simply add $1$ to both
                coordinates.  Therefore, when comparing the statement
                in this theorem with other results, it is important to
                bear this in mind.  We use Serre values.  The equation
                $x^p-\pi_K x=\pi_K$ is used in \cite[\S1.4 Example
                  2]{lubin:2013} with $K=\bQ_p$ and $\pi_K=p$ where
                the ramification number is reported as $p/(p-1)$.
                This is an Artin value. Using Theorem \ref{extensions}
                with $f=t=1$, $e=d=p-1$, $b=-1$ and
                $\alpha=\beta=\pi_K$, the Serre value is
                $1/(p-1)$. The difference between the two values is
                $1$.
\end{remark}

The rest of this section is concerned with the proof of
Theorem \ref{extensions}. We
begin with an exercise in group theory. Since the residue field $\kappa$ is
finite, the group $G=\Gal(M/K)$ for the Galois closure $M/K$ of $L/K$ is
solvable \cite[IV \S2 Corollary 5]{serre:local}.
Any solvable transitive subgroup $G$ of the symmetric group
$S_p$ on $p$ letters contains a unique subgroup $\langle\sigma\rangle$ of order $p$ and
is contained in the normalizer of $\langle\sigma\rangle$ in $S_p$ ({\em e.g.}
\cite[pg.~638 Exercise 20]{dummit:foote}).  
Note that $\mbox{Gal}(M/K)/\langle\sigma\rangle$
cyclic of order $d$ for some $d\mid (p-1)$.  Let $M'$ be the fixed field
of $\langle\sigma\rangle$, a cyclic extension of $K$ of degree $d$.  Let
$\langle\tau\rangle$ be the subgroup that fixes $L$.  From this it
follows that there is an integer $r$ of order $d$ modulo $p$ 
such that $G$ is as in the statement above.
At this point, the elements $\sigma,\tau\in G$ along with
the integers $d,r$ are fixed.  We characterize $M'$.

Since the residue field $\kappa$ contains $\bF_p^\times$, $K$ contains
the $d$th roots of unity. Thus $M'/K$ is Kummer and $M'=K(y)$ with
$y^{d}=\alpha$ for some $\alpha\in K$ representing a coset of order
$d$ in the quotient group $K^\times/(K^\times)^{d}$, and
$\tau(y)=\rho y$.  Within $M'$ there is a maximally unramified
extension of $K$, which we call $K'$.  Let $e=[M':K']$ and $f=[K':K]$.
Thus $d=ef$.
Let $\pi_K$, $\pi_{K'}$ denote prime elements in $K$, $K'$,
respectively.  We can replace $y$ by $y\pi_K^i$ and still have a
Kummer generator for $M'/K$, and so we can assume that $0<
v_{M'}(y)\leq e$.  Since $M'/K'$ is totally ramified and tame
(including the case $e=1$ where $M'=K'$), $M'=K'(z)$ where
$z^e=\pi_{K'}$. Since $K'/K$ is unramified, $\pi_{K'}=\pi_Ku$ for some
$u\in\euO_{K'}^\times$. Since both $y$ and $z$ are Kummer generators
of $M'/K'$ we have $y=z^t\omega$ for some $1\leq t\leq e$ satisfying
$\gcd(t,e)=1$, and $\omega\in \euO_{K'}^\times$. Thus
$y^e=\pi_K^tu^t\omega^e$.  Let $\omega'=u^t\omega^e\in
\euO_{K'}^\times$.  Since $K'/K$ is an unramified Kummer extension,
$K'=K(v)$ where $v^f=\mu\in \euO_K^\times$ where $\mu$ represents a
coset of order $f$ in the quotient group
$\kappa^\times/(\kappa^\times)^f$.  But $\omega'=y^e/\pi_K^t$ is also
a Kummer generator for $K'$, so $\omega'=v^s\gamma$ for some $1\leq
s<f$ satisfying $\gcd(s,f)=1$, and $\gamma\in \euO_K^\times$.  As a
result, $y^{d}=(y^e)^f=\pi_K^{tf}\gamma^f\mu^s$.  But then, without any
loss of generality, we can replace $\mu$ by $\mu^s$ and relabel, since
the descriptions of these two elements are the same.  Now for the
converse, observe that for $y^{d}=\alpha$ with $\alpha$ as above,
$y^e/(\pi_K^t\gamma)$ satisfies the equation $v^f=\mu$ and thus
generates an unramified extension of degree $f$. Furthermore, $y$
satisfies $y^e=\pi_K^t\gamma v\in K'$. Since $\gcd(t,e)=1$, let
$t't\equiv 1\bmod e$. Then $y^{t'}$ satisfies $x^e\in \pi_{K'}(K')^e$
for some prime element $\pi_{K'}\in K'$. In summary, we have found
that $M'=K(y)$ where $y^{d}=\alpha=\pi_K^{tf}\gamma^f\mu^s$ as in
the statement of the theorem.
Note that $v_{M'}(y)=t$.

Consider the cyclic extension $M/M'$, which is ramified because $L/K$
is ramified of degree $p$, and ramification is multiplicative in
towers.  Assume for a contradiction that $M/M'$ is not typical. So
$M=M'(X)$, $X^p=\pi_{M'}$ and $\zeta_p\in M'$ where $\zeta_p$ is a
primitive $p$th root of unity, then the norm
$z=N_{M/L}(X)=\prod_{i=0}^{d-1}\tau^i(X)\in L$ satisfies
$z^p=N_{M/L}(X^p)=\prod_{i=0}^{d-1}\tau^i(\pi_{M'})\in K$.  So $z^p\in
K$ where $v_{M'}(z^p)=d$, which means that $v_K(z^p)=f$. Since
$\gcd(f,p)=1$, this means that we can generate $L$ by a $p$th root of
a prime element, contradicting our assumption that $L/K$ is typical.
We conclude that $M/M'$ is a typical Galois extension, which means that $M=M'(X)$ where
$X^p-X=\beta'$ for some $\beta'\in M'$ with $v_{M'}(\beta')=-b'$,
$1\leq b'<v_M(p)/(p-1)$ with $p\nmid b'$ and $(\sigma-1)X= 1 +
\varepsilon$ where in $\chr(K)=p$ we have $\varepsilon=0$. In
$\chr(K)=0$ we have $\varepsilon\in M$ with
$v_M(\varepsilon)=v_M(p)-(p-1)b'>0$ \cite[Chap.~3,
  \S2]{fesenko:vostokov}.

Let $\bX=yX$. Observe that $v_M(\bX)=pt-b'$, and set
\begin{equation}\label{x}
x=\frac{1}{d}(1+\tau+\cdots +\tau^{d-1})\bX\in L.
\end{equation}
Let $G_i$ be the
ramification filtration for $\Gal(M/K)$, then
$G_i\cap\langle\sigma\rangle$ yields the ramification filtration for
$\Gal(M/L)=\langle\tau\rangle$. As a result the maximal unramified
extension of $L$, called $L'$, satisfies $[M:L']=e$ and
$[L':L]=f$. The different for $M/L$, $\euD_{M/L}$, is $\euP_M^{e-1}$.
Let $\Tr=1+\tau+\cdots+\tau^{d-1}$ denote the trace for $M/L$.
Since $\Tr$ is $\euO_L$-linear, $\Tr(\euP_M^n)$ is an ideal in
$\euO_L$. For any integer $r$, \cite[Ch.~3, Prop.~7]{serre:local}
shows that $\Tr(\euP_M^n)\subseteq \euP_L^r$ if and only if
$\euP_M^n\subseteq \pi_L^r\euD_{M/L}^{-1}=\euP_M^{er-e+1}$. Thus
$r\leq (n+e-1)/e$, and so $\Tr(\euP_M^n)= \euP_L^r$ where
$r=1+\lfloor(n-1)/e\rfloor=\lceil n/e\rceil$.  This proves that
$v_L(x)=\lceil (pt-b')/e\rceil$.  Let us set
$$b=-\left\lceil \frac{pt-b'}{e}\right\rceil.$$ So
$b'=eb+pt+r$ for some $0\leq r<e$, and $v_M(x/y)=-eb-pt=r-b'\geq -b'= v_M(X)$.

In the next step of this argument, we will identify an element
$\lambda\in L$ such that $L=K(\lambda)$, $v_L(\lambda)=-b$,
$(\sigma-1)(\lambda/y)\in 1+\euP_M$, and
$\lambda^p-\alpha^{(p-1)/d}\lambda\in K$.  In $\chr(K)=p$, this is
$\lambda=x$. In $\chr(K)=0$, the process is more complicated, and thus
the two arguments will diverge. But before they diverge, 
observe that
as soon as we identify an element $\lambda\in L$ such that
$(\sigma-1)(\lambda/y)\in 1+\euP_M$,  
we may conclude that $r=0$. Indeed,
$$b'=eb+pt \mbox{ and }p\nmid b.$$ 
The reason is that
$v_M((\sigma-1)\mu)\geq v_M(\mu)+b'$ for all $\mu\in M$, and if
$r\neq 0$, then
$v_M(\lambda/y)>v_M(X)=-b'$, which would imply
$(\sigma-1)(\lambda/y)\in\euP_M$.
Additionally, as soon as we prove
$x^p-\alpha^{(p-1)/d}x=\beta$ for some $\beta\in K$, we can observe that since $0<b'=eb+pt$, we have $-b<pt/e$, which means that
$v_L(x^p)<v_L(\alpha^{(p-1)/d}x)$ and thus $v_K(\beta)=-b$.
Note that $p\nmid b$ and 
$$0<b+pt/e< pv_K(p)/(p-1).$$

\subsection*{Determining $\lambda$ for $\chr(K)=p$}
Assume $\chr(K)=p$.
For $1\leq i<d$, check that
$\sigma\tau^i=\tau^i\sigma^{r^{-i}}$. 
Since $\rho=r$, using \eqref{x} we have
\begin{multline*}
\sigma x=\frac{\sigma}{d}(1+\tau+\cdots \tau^{d-1})\bX=
\frac{1}{d}(\sigma+\tau\sigma^{r^{-1}}+\tau^2\sigma^{r^{-2}}+\cdots+ \tau^{d-1}\sigma^{r^{-(d-1)}})\bX\\=
\frac{1}{d}((\bX+y)+\tau(\bX+r^{-1}y)+\tau^2(\bX+r^{-2}y)+\cdots+ \tau^{d-1}(\bX+r^{-(d-1)}y))=x+y,
\end{multline*}
which means that $(\sigma-1)x=y$.
Therefore $\sigma^ix=x+iy$ for $0\leq i<d$.
Now the norm of $x$, namely $N_{M/M'}(x)=\prod_{i=0}^{p-1}\sigma^ix$ equals
$$\prod_{i=0}^{p-1}(x+iy)=y^p\prod_{i=0}^{p-1}\left(\frac{x}{y}+i\right)
=y^p\left(\frac{x^p}{y^p}-\frac{x}{y}\right)=x^p-\alpha^{(p-1)/d}x.$$
Clearly $x^p-\alpha^{(p-1)/d}x$, as a norm, is fixed by $\sigma$, but because
$\alpha^{(p-1)/d}\in K$, it is also fixed by $\tau$.  As a result,
$x^p-\alpha^{(p-1)/d}x\in K$.

\subsection*{Determining $\lambda$ for $\chr(K)=0$}
In this case, $x$, as defined by \eqref{x}, only provides us with a first approximation for $\lambda$. We will set $\lambda_0=0$, $\lambda_1=x\in L$,
and construct a sequence
$\{\lambda_n\}\subset L$ satisfying certain properties such that 
$\lambda=\lim\lambda_n$ gives us the desired element. First, we need three preliminary results. 
Observe that
$v_L(\lambda_1)=v_L(x)=-b$.

\begin{lemma} \label{tau-x}
$(\sigma-1)x =y+y\Delta_1\in M$
where $v_M(\Delta_1)>b'-(eb+pt)\geq 0$.
\end{lemma}
\begin{proof}
Recall that $M=M'(X)$ where $X$ satisfies an Artin-Schreier equation and $(\sigma-1)X=1+\varepsilon$ where $v_M(\varepsilon)=v_M(p)-(p-1)b'$. Furthermore,
$\bX=yX$ and thus
$\sigma^i\bX=\bX+iy+(1+\sigma+\cdot+\sigma^{i-1})y\varepsilon
\equiv \bX+iy+iy\varepsilon\bmod y\varepsilon\euP_M^{b'}$ for $0\leq
i<p$.  In $\chr(K)=0$, we don't have $\rho= r$, but we do have
$\rho\equiv r\bmod p$.  So for $0\leq j<d$, given $r^{-j}$, we may
define $\bar{r}_j\equiv r^{-j}\bmod p$ with $0\leq \bar{r}_j<p$.
This means that $\tau^j\sigma^{r^{-j}}\bX=\tau^j\sigma^{\bar{r}_j}\bX
\equiv \tau^j\bX+
\rho^{j}\bar{r}_jy+\tau^j\bar{r}_jy\varepsilon
\bmod y\varepsilon\euP_M^{b'}$  where $\rho^{j}\bar{r}_j\equiv 1\bmod p$.
Since $v_M(p)\geq v_M(\varepsilon)+b'$, we find  that
$\tau^j\sigma^{r^{-j}}\bX\equiv \tau^j\bX+y+y\tau^j\varepsilon
\bmod y\varepsilon\euP_M^{b'}$. Therefore
\begin{multline*}
\sigma x=\frac{\sigma}{d}(1+\tau+\cdots \tau^{d-1})\bX=
\frac{1}{d}(\sigma+\tau\sigma^{r^{-1}}+\tau^2\sigma^{r^{-2}}\cdots \tau^{d-1}\sigma^{r^{-(d-1)}})\bX\\\equiv
x+\frac{1}{d}\sum_{j=0}^{d-1}y
+\frac{y}{d}\sum_{j=0}^{d-1}\tau^j\varepsilon
\equiv x+y+\frac{y}{d}\Tr(\varepsilon)
\bmod y\varepsilon\euP_M^{b'},
\end{multline*}
where $\Tr$ is the trace for $M/L$. Recall that $\Tr(\euP_M^n)=\euP_L^{\lceil n/e\rceil}$.
Since
$e\mid v_M(p)-(p-1)b'=v_M(\varepsilon)$
and $v_M(p)>(p-1)b'$, this means that $v_M(\Tr(\varepsilon))=v_M(\varepsilon)=v_M(p)-(p-1)b'\geq e$.
We have proven that $(\sigma-1)x
=y+y\Delta_1$ for some $\Delta_1\in M$ where $v_M(\Delta_1)=
v_M(p)-(p-1)b'\geq e>r=b'-(eb+pt)$.
\end{proof}

Define
$$\wp_{\alpha}(X)= y((X/y)^{p}-X/y)=
\frac{1}{y^{p-1}}X^p-X=\frac{1}{\alpha^{(p-1)/d}}X^p-X\in K[X].$$

\begin{lemma}\label{wp_{alpha}}
$v_M((\sigma-1)\wp_{\alpha}(x)>b'-eb$.
\end{lemma}
\begin{proof}
Using Lemma \ref{tau-x}, we have 
\begin{multline*}
\frac{1}{y}(\sigma-1)\wp_{\alpha}(x)=
\left(\frac{x}{y}+1+\Delta_1\right)^{p}-\left(\frac{x}{y}+1+\Delta_1\right)-\left(\left(\frac{x}{y}\right)^{p}-\frac{x}{y}\right)\\
=\left(\frac{x}{y}+1+\Delta_1\right)^{p}-\left(\frac{x}{y}\right)^{p}-(1+\Delta_1)\\
=\sum_{i=1}^{p-1}\binom{p}{i}\left(\frac{x}{y}\right)^i(1+\Delta_1)^{p-i}+
\sum_{i=1}^{p-1}\binom{p}{i}\Delta_1^i+(\Delta_1^p-\Delta_1).
\end{multline*}
Multiplying back through by $y$, it is enough to show that
$v_M(py(x/y)^{p-1})\geq b'-eb$ when $v_M(x/y)\leq 0$, and $v_M(px)\geq
b'-eb$ when $v_M(x/y)>0$, while also showing $v_M(y\Delta_1)\geq b'-eb$.  Under
$v_M(x/y)\leq 0$, $v_M(py(x/y)^{p-1})\geq b'-eb$ is equivalent to
$v_M(p)\geq b'+(p-2)(be+pt)$, which follows from $v_M(p)>(p-1)b'$ and
$b'\geq be+pt$. Under $v_M(x/y)>0$, $v_M(px)\geq b'-eb$ follows from
$v_M(p)>(p-1)b'\geq b'$. This leaves $v_M(y\Delta_1)> b'-eb$, which is
equivalent to $v_M(\Delta_1)> b'-(eb+pt)=r$ and follows from
Lemma \ref{tau-x}.
\end{proof}
We require one more lemma, which is a generalization of \cite[(2.2) Lemma]{fesenko:vostokov}.

\begin{lemma}\label{fes-lemma}
Given $Y\in L\setminus K$ there is an $y\in K$ such that $v_M((\sigma-1)Y)=v_M(Y-y)+b'$.
\end{lemma}
\begin{proof}
Let $\pi_L\in L$ be a prime element, and express $Y=\sum_{i=0}^{p-1}a_i\pi_L^i$
for some $a_i\in K$. For $1\leq i<p$, $p\nmid v_M(\pi_L^i)$, and thus
$v_M((\sigma-1)\pi_L^i)=v_M(\pi_L^i)+b'$.
Let $y=a_0$. Then
$v_M((\sigma-1)Y)=v_M((\sigma-1)(Y-y)=
v_M((Y-y)+b'$.
\end{proof}

We are now prepared to follow the argument from
\cite[pg.~76]{fesenko:vostokov} by constructing a sequence
$\{\lambda_n\}\subset L$ that satisfies the following conditions
\begin{equation}\label{goal}
\begin{array}{c}
v_L(\lambda_n)=-b,\quad v_L(\lambda_{n+1}-\lambda_n)\geq v_L(\lambda_n-\lambda_{n-1})+1,\\
v_M((\sigma-1)\wp_\alpha(\lambda_{n+1}))\geq v_M((\sigma-1)\wp_\alpha(\lambda_n))+1,
\end{array}
\end{equation}
with $\wp_\alpha$ defined in Lemma
\ref{wp_{alpha}}.
Once we have done this, we will set $\lambda=\lim\lambda_n$, and observe that
$v_L(\lambda)=-b$ and
$\wp_\alpha(\lambda)\in K$.
To do this, we will define an auxiliary sequence
$$\delta_n=(\sigma-1)\wp_\alpha(\lambda_n).$$ Using Lemma
\ref{wp_{alpha}}, $v_M(\delta_1)>b'-eb$, which since $b'\geq eb+pt$,
also means that $v_M(\delta_1)>pt=v_M(y)$.  If we ever have
$\delta_n=0$ then since $\wp_{\alpha}(\lambda)\in K$. Simply set
$\lambda=\lambda_n$.
This means that we can assume throughout the argument we can assume that $\delta_n\neq 0$, and by
induction that $v_L(\lambda_n)=-b$, $v_M(\delta_n)>b'-eb\geq
pt=v_M(y)$, and $(\sigma-1)\lambda_n=y+y\Delta_n$ where
$v_M(\Delta_n)>b'-(eb+pt)$.  

Using Lemma \ref{fes-lemma},
there is a $c_n\in K$ such that
$v_M(\delta_n)=v_M(\wp_\alpha(\lambda_n)+c_n)+b'$.  Put
$\mu_n=\wp_\alpha(\lambda_n)+c_n\in L$ and set
$$\lambda_{n+1}=\lambda_n+\mu_n\in L.$$

Record that $v_M(\mu_n)=v_M(\delta_n)-b'$ and that
$(\sigma-1)\mu_n=\delta_n$. Since $v_M(\mu_n)>-eb$, the first statement
means that $v_L(\lambda_{n+1})=-b$. Using the second statement,
$(\sigma-1)\lambda_{n+1}=(\sigma-1)\lambda_n+\delta_n$. Thus
$(\sigma-1)\lambda_{n+1}=y+y\Delta_{n+1}$ where
$y\Delta_{n+1}=y\Delta_n+\delta_n$. Therefore,
$v_M(\Delta_{n+1})>b'-(eb+pt)$. Since
$v_L(\lambda_{n+1}-\lambda_n)=v_L(\mu_n)$ and
$v_M(\mu_n)=v_M(\delta_n)-b'$, all that remains of
\eqref{goal} to be verified, is that $v_M(\delta_{n+1})\geq
v_M(\delta_n)+1$, and this is the next result.
\begin{lemma}
$$\delta_{n+1}=(\sigma-1)\wp_{\alpha}(\lambda_{n+1})=
(\sigma-1)\wp_{\alpha}(\lambda_n+\mu_n)
\equiv 0\bmod \delta_n\euP_M.$$
\end{lemma}
\begin{proof}
Using the definition of $\delta_n$, this is the same as proving that
$$(\sigma-1)\big(\wp_{\alpha}(\lambda_n+\mu_n)-\wp_{\alpha}(\lambda_n)\big)
=(\sigma-1)y^{1-p}\sum_{i=1}^{p-1}\binom{p}{i}\lambda_n^i\mu_n^{p-i}+(\sigma-1)\wp_{\alpha}(\mu_n)
\equiv 0\bmod \delta_n\euP_M.$$
There are two summands to consider.
Consider the first. Note that
$$v_M\left((\sigma-1)y^{1-p}\sum_{i=1}^{p-1}\binom{p}{i}\lambda_n^i\mu_n^{p-i}\right)\geq b'-(p-1)pt+v_M(p)-(p-1)eb+v_M(\mu_n).$$
Since $v_M(p)>(p-1)b'\geq (p-1)(eb+pt)$ and $v_M(\mu_n)= v_M(\delta_n)-b'$, it follows that  the first summand is
$0\bmod \delta_n\euP_M$.
Consider the second.
Note that 
$$(\sigma-1)\wp_{\alpha}(\mu_n)=\wp_{\alpha}(\mu_n+\delta)-\wp_{\alpha}(\mu_n)=
y^{1-p}\sum_{i=1}^{p-1}\binom{p}{i}\mu_n^i\delta_n^{p-i}+\wp_{\alpha}(\delta_n).$$
For $1\leq i\leq p-1$,
$v_M(y^{1-p}\binom{p}{i}\mu_n^i\delta_n^{p-i})=v_M(p)-(p-1)pt+pv_M(\delta_n)-ib'\geq
v_M(p)-(p-1)(pt+eb)+(p-1-i)b'+v_M(\delta_n) \geq
v_M(p)-(p-1)(pt+eb)+v_M(\delta_n) \geq
v_M(p)-(p-1)b'+v_M(\delta_n)>v_M(\delta_n)$. Furthermore, since
$v_M(\delta_n/y)>0$, we also have $\wp_{\alpha}(\delta_n)\equiv 0\bmod
\delta_n\euP_M$.
\end{proof}

We have proven that there is a $\lambda\in L$ such that $v_L(\lambda)=-b$,
$(\sigma-1)(\lambda/y)\in 1+\euP_M$, and
$$\lambda^p-\alpha^{(p-1)/d}\lambda\in K.$$

\subsection{Ramification}

We turn now to the ramification break number for  the extensions described in
Theorem \ref{extensions}. 
The ramification number for $M/M'$ is
$b'=eb+pt$.
The Herbrand function for $M/K$,
using numbering as in \cite{serre:local} and graph from
 \cite[page
  116]{artin:tate},
contains a segment of slope
$1/e$ from the origin to
$x=eb+pt$. Then a ray of slope $1/(ep)$. 
We are interested in the vertex $(eb+pt,b+pt/e)$,
as it gives the largest upper ramification number for $M/K$ as
$b+pt/e$, and thus gives $b+pt/e$ as the upper ramification number
for $L/K$. Since the Herbrand function for $L/K$ has a segment of 
slope $1$ to $y=b+pt/e$, followed by a ray of slope $1/p$, the upper and lower ramification numbers for $L/K$ agree and equal
\begin{equation}\label{ram-result}
\ell=b+pt/e.
\end{equation}
Unless $e=1$,
$\ell$ is not an integer.

\subsection{Different}
Using the fact that $\euD_{M/K}=\euD_{M/L}\euD_{L/K}$ \cite[III \S4
  Proposition 8]{serre:local} along with the formula for the exponent
on the different in the Galois case, namely \cite[IV \S Proposition
  4]{serre:local}, we see that
$\euD_{M/K}=\euP_M^{(ep-1)+(eb+pt)(p-1)}$ and
$\euD_{M/L}=\euP_M^{e-1}$. Therefore
$$\euD_{L/K}=\euP_L^{(b+1)(p-1)+pt\frac{(p-1)}{e}}.$$ Thus the
same expression for the exponent of the different in terms of the
ramification number for the extension holds regardless of whether $L/K$ is Galois or
not.

\section{Hopf-Galois module structure}

Greither and Pareigis have classified the finitely many Hopf-Galois
structures that are possible for a given separable extension
\cite{greither:pareigis}.  Childs has showed that there is only one
such structure when we restrict to separable extensions $L/K$ of
degree $p$ \cite[\S2]{childs:separable}, which means that there is
only one for typical extensions.  While Childs assumes $\chr(K)=0$,
his argument applies equally well in $\chr(K)=p$.  Here, we provide a
sketch of \cite[\S2]{childs:separable}, relaying on \cite{childs:book}
for some of the details.  The Hopf algebra $\caH$ that provides the
unique Hopf-Galois structure is described by descent. When the
extension is Galois, $\caH$ is just the group algebra $K[G]$.
In any case, Theorem \ref{extensions} then
allows a simple, explicit description of the action of $\caH$ on
$L/K$. Without any adjustment, a scaffold exists for this action.  We
close with a discussion of what this means for the $\caH$-Galois
module structure of the ideals of $L$.

\subsection{Hopf-Galois structure}
\begin{theorem}[Childs] \label{hopf-gal-structure}
Adopt the notation of Theorem \ref{extensions} for a given a typical
extension $L/K$. Recall that $d\mid p-1$. Let $ds=p-1$, and let $r_0$
denote a primitive root modulo $p$ with Teichm\"{u}ller representative
$\rho_0$ such that $r_0^s\equiv r\bmod p$ and $\rho_0^s=\rho$. Set
$\Psi=-1/y\sum_{k=0}^{p-2}\rho_0^{-k}\sigma^{r_0^{k}}$.  Then the
unique Hopf algebra $\caH$ such that $L/K$ is a $\caH$-Galois
extension is explicitly
$$\caH=K[\Psi].$$ It is contained in the group ring $K[y][\langle\sigma\rangle]$
and inherits is counit $\varepsilon(\Psi)=0$, antipode $S(\Psi)=-\Psi$ and antipode $\Delta$ from $K[y][\langle\sigma\rangle]$. For example,
in $\chr(K)=p$, explicitly
$$\Delta(\Psi)=\Psi\otimes 1+1\otimes \Psi -\alpha^s\sum_{i=1}^{p-1}\frac{1}{p}\binom{p}{i}\Psi^i\otimes \Psi^{p-i}.$$
\end{theorem}
\begin{proof}
As Childs explains in \cite[\S2]{childs:separable}, the unique Hopf
algebra $\caH$ is described by descent.  Using our notation, the group
algebra $M'[\langle\sigma\rangle]$ where $M'=K(y)$ has $K$-basis
$\{y^i\sigma^j:0\leq i< d, 0\leq j<p\}$.  The action of
$\langle\tau\rangle$ on these basis elements is given by
$\tau^k(y^i\sigma^j)=(\rho^{k}y)^i\sigma^{jr^k}=(\rho_0^{sk}y)^i\sigma^{jr_0^{sk}}$
with the Hopf algebra
$\caH=M'[\langle\sigma\rangle]^{\langle\tau\rangle}$ determined to be
the sub-algebra of $M'[\langle\sigma\rangle]$ fixed by $\tau$. The
counit $\varepsilon$, antipode $S$ and comultiplication $\Delta$ for
$\caH$ are determined in $M'[\langle\sigma\rangle]$.

Given a basis element for $M'[\langle\sigma\rangle]$, the sum over its
orbit under $\langle\tau\rangle$ certainly lies in
$M'[\langle\sigma\rangle]^{\langle\tau\rangle}$.  
There is only one element in the orbit of $y^0\sigma^0$, namely $1$. 
The orbit of $y^i\sigma^0$ for $i\neq 0$ is a sum of
$d$th roots of unity that equals zero.
Otherwise, we are considering the orbit generated by
$y^i\sigma^j$ where $j\neq 0$ represents some coset of $\bF_p^\times/\langle
r\rangle$. A
complete set of coset representatives for $\bF_p^\times/\langle
r\rangle$ is given by $\{r_0^t:0\leq t< s\}$.  
And so we are considering the
 orbit of $y^i\sigma^{r_0^t}$, namely
$$y^i\theta(i,t)=y^i\sum_{k=0}^{d-1}\rho_0^{isk}\sigma^{r_0^{t+sk}}.$$ These orbits biject with
$\{(i,t): 0\leq i< d, 0\leq t< s\}$, a set with
$ds=p-1$
elements.  Together with $1$, we have $K$-basis of dimension $p$ for $\caH$.

We would like now, as in \cite[\S2]{childs:separable}, to perform a
change in basis. First, we introduce, mechanically, the basis change
from \cite[\S2]{childs:separable}. Second, we motivate everything
based upon \cite[\S16]{childs:book}.  
Observe that $\theta(i,t)=\theta(i+bd,t)$
for all $b\in \bZ$, and for $0\leq i<p$, let
$$\Theta(i)=\sum_{t=0}^{s-1}\rho_0^{it}\theta(i,t)=\sum_{k=0}^{p-2}\rho_0^{ik}\sigma^{r_0^{k}}.$$
The idea is to replace, for a fixed $i$ in $0\leq i< d$, the $s$
elements $\{y^i\theta(i,t):0\leq t<s\}$ in our basis with the
alternate $s$ elements $\{y^i\Theta(i+bd):0\leq b<s\}$. Since
$y^{i+bd}=\alpha^by^i\in K^\times y^i$, this is the same as replacing
them with $\{y^{i+bd}\Theta(i+bd):0\leq b<s\}$.  Clearly,
$\{y^i\Theta(i+bd):0\leq b<s\}$ is contained in the $K$-span of
$\{y^i\theta(i,t):0\leq t<s\}$. Furthermore since
$\sum_{b=0}^{s-1}\rho_0^{(t-a)(i+bd)}=s\delta_{t,a}$ where
$\delta_{t,a}$ is the Kronecker delta function, we have
$sy^i\theta(i,a)=\sum_{b=0}^{s-1}\rho_0^{-a(i+bd)}y^i\Theta(i+bd)$ and thus find that
the $K$-spans are equal. This means that 
$\{1\}\cup \{y^i\Theta(i):0\leq i\leq p-2\}$ is a $K$-basis for $\caH$.
Since $\sum_{k=0}^{p-2}\rho_0^{ik}=0$ unless $(p-1)\mid i$, we see
that $\{y^i\Theta(i):1\leq i\leq p-2\}$ lies within the augmentation
ideal $\caH^+=\{h\in\caH:\varepsilon(h)=0\}$. Furthermore,
$\varepsilon(\Theta(0))=(p-1)$, thus $\Theta(0)-(p-1)\in \caH^+$ as well. We
now adjust Childs' basis very slightly to one more amenable
to our purposes. Set $j=p-i-1$ and for $1\leq i<p-1$, set
$$\Psi_j=-\frac{y^i\Theta(i)}{\alpha^s}=\frac{-1}{y^j}\sum_{k\in \bZ/(p-1)\bZ}\rho_0^{-jk}\sigma^{r_0^{k}},$$
and additionally,  $\Psi_{p-1}=-(\Theta(0)-(p-1))/y^{p-1}$.
Thus $\{\Psi_j:1\leq j\leq p-1\}$ is a $K$-basis for $\caH^+$.

We now use \cite[\S16]{childs:book} to explain this choice of basis,
and find, as a result of this explanation, that $\caH=K[\Psi_1]$. In
\cite[\S16]{childs:book}, a homomorphism is defined from
$\bF_p^\times$ to the group of Hopf algebra automorphisms of
$\caH$. Let $\chi$ be the identity map in $\chr(K)=p$, and the
Teichm\"{u}ller character such that the primitive root $r_0\in
\bF_p^\times$ maps to $\rho_0\in \bZ_p$ in $\chr(K)=0$.  Given $m\in
\bF^\times$, the automorphism is denoted by $[m]$. It is proven that
$\caH^+$ is a $\bZ_p[\bF_p^\times]$-submodule of $\caH$ in $\chr(K)=0$
or an $\bF_p[\bF_p^\times]$-submodule in $\chr(K)=p$ \cite[Lemma
  16.2]{childs:book}. In either case, the idempotent elements of the
group ring decompose $\caH^+\cong \oplus_{j=1}^{p-1}\caH_j$ into
one-dimensional $K$-spaces $\caH_j=\{h\in\caH^+:[m](h)=\chi^j(m)h\}$,
an eigenspace decomposition. Since $[m](\sigma)=\sigma^m$, one can
check that $\caH_j=K\Psi_j$, which explains the significance of the
basis that we have chosen. Let $x_i=y^i\Psi_i$ so that $x_i$ agrees
with notation in \cite[\S16]{childs:book}.  The argument leading to
\cite[Proposition 16.5]{childs:book} proves that $K[x_1]$ equals the
$K$-span on $\{1,x_1,\ldots,x_{p-1}\}$. This implies $K[\Psi_1]=\caH$
as well, and so for the statement in the theorem, set $\Psi=\Psi_1$.

In $\chr(K)=p$, it is easy to show that $x_1^i=i!x_i$ for $1\leq i<p$,
and thus this is something we do in Lemma \ref{w_i-in-char-p}.  As a
result, using the formula for comultiplication in
\cite[(16.7)]{childs:book}, the formula for comultiplication
$\Delta(\Psi)$ in the statement in the theorem follows. In
$\chr(K)=0$, there are units $w_i\in\bZ_p$ such that $x_1^i=w_ix_i$
that do not a simple description, and thus we leave the formula for 
$\Delta(\Psi)$ implicit.
\end{proof}

\begin{lemma}\label{w_i-in-char-p}
Let $\bZ_{(p)}$ be the integers localized at $p$. Then 
For $1\leq i\leq p-1$,
$$\left(-\sum_{k=1}^{p-1}\frac{1}{k}x^k\right)^{\!\!i}\equiv -i!\sum_{k=1}^{p-1}\frac{1}{k^i}x^k\mod(p,x^p-1)$$
in the polynomial ring $\bZ_{(p)}[x]$.
\end{lemma}
\begin{proof}
Since $r_0$ is a primitive root modulo $p$, 
$\sum_{k=1}^{p-1}r_0^{ek}\equiv 0\bmod p$, for any exponent
$1\leq e\leq p-2$.
This means that
$\sum_{t=2}^{p-1}\frac{1}{t^e}=\sum_{k=1}^{p-2}r_0^{-ek}\equiv -1\bmod p$.
It is easy to see that
$\frac{1}{t^i(1-t)}=\frac{1}{1-t}+\sum_{e=1}^i\frac{1}{t^e}$. Thus
$\sum_{t=2}^{p-1}\frac{1}{t^i(1-t)} =
\sum_{t=2}^{p-1}(\frac{1}{1-t}+\sum_{e=1}^i\frac{1}{t^e})=
\sum_{t=2}^{p-1}(\frac{1}{1-t}+\frac{1}{t})+
\sum_{e=2}^i\sum_{t=2}^{p-1}\frac{1}{t^e}
= (\frac{1}{p-1}-1)
+\sum_{e=2}^i-1\equiv -(i+1)\bmod p$.
Let $t\equiv k/m\bmod p$. This identity becomes
$\sum \frac{m^i}{k^{i}}\frac{m}{m-k} \equiv -(i+1) \bmod p$,
where the left-hand-sum is over all $1\leq k\leq p-1$ except $k=m$.
This means that
$\sum \frac{1}{k^{i}}\frac{1}{m-k} \equiv \frac{-(i+1)}{m^{i+1}} \bmod p$,
which allows us to prove  by induction that
for $1\leq i\leq p-2$,
$$\left(\sum_{k=1}^{p-1}\frac{1}{k^i}x^k\right)\left(\sum_{k=1}^{p-1}\frac{1}{k}x^k\right)
\equiv -(i+1)\sum_{k=1}^{p-1}\frac{1}{k^{i+1}}x^k\mod(p,x^p-1).$$
From this the result follows.
\end{proof}

\subsection{Hopf-Galois module structure}
Based upon Theorem \ref{hopf-gal-structure}, 
$\caH=K[\Psi]$ is the unique Hopf algebra that
makes the typical extension $L/K$ Hopf-Galois.

\begin{theorem}\label{scaffold}
Let $L=K(x)$ be a typical extension of $K$, with $x$ as in Theorem \ref{extensions} and ramification number $\ell$. Then
$\Psi\cdot 1=0$ and
for $1\leq i\leq p-1$, 
$\Psi\cdot x^i\in L$. In particular,
$$\begin{array}{lc}
\Psi\cdot x^i= ix^{i-1} & \mbox{in }\chr(K)=p,\\
\Psi\cdot x^i\equiv ix^{i-1}\bmod x^{i-1}\euP_L^{v_L(p)-(p-1)\ell}& \mbox{in }\chr(K)=0.
\end{array}$$
\end{theorem}
\begin{proof}
Recall that $\sigma$ is an automorphism of $M/K$.
Since $\sum_{k=0}^{p-2}\rho_0^{-k}=0$, $\Psi\cdot 1=0$.
Because the argument is much simpler for
$\chr(K)=p$, we treat it first.
Note
$\sigma^i x=x+iy$ and $\rho_0=r_0$. Thus
$\Psi\cdot x^i=\frac{-1}{y}\sum_{k=0}^{p-2}r_0^{-k}(x+r_0^ky)^i=
\frac{-1}{y}\sum_{k=0}^{p-2}\sum_{t=0}^i\binom{i}{t}x^{i-t}r_0^{(t-1)k}y^t=
-\sum_{t=0}^i\binom{i}{t}x^{i-t}y^{t-1}\sum_{k=0}^{p-2}r_0^{(t-1)k}
=
\sum_{t=0}^i\binom{i}{t}x^{i-t}y^{t-1}\delta_{t,1}=ix^{i-1},$
where $\delta_{i,j}$ is the Kronecker delta function.
In $\chr(K)=0$, $\sigma x=x+y+y\Delta$ where $\Delta\in M$ with 
$v_M(\Delta)=v_M(p)-(p-1)(be+pt)$, we need to introduce further notation. Let
 $1\leq r_k<p$ satisfy
$r_k\equiv r_0^k\bmod p$ and set
$\Delta_k=(1+\sigma+\cdots+\sigma^{r_k-1})\Delta$, we have
$\sigma^{r_0^k}=\sigma^{r_k}=x+y(r_k+\Delta_k)$ for $1\leq k\leq p-2$.
As a result,
\begin{multline*}
\Psi\cdot x^i=\frac{-1}{y}\sum_{k=0}^{p-2}\rho_0^{-k}(x+y(r_k+\Delta_{k}))^i=
\frac{-1}{y}\sum_{k=0}^{p-2}\rho_0^{-k}\sum_{s=0}^i\binom{i}{s}x^{i-s}y^s(r_k+\Delta_{k})^s\\=
\frac{-1}{y}\sum_{s=0}^i\binom{i}{s}x^{i-s}y^s\sum_{k=0}^{p-2}\rho_0^{-k}(r_k+\Delta_{k})^s.\end{multline*}
Since $\rho_0$ is a primitive $p-1$ root of unity, $\sum_{k=0}^{p-2}\rho_0^{-k}(r_0^k+\Delta_{r_0^k})^s=0$ for $s=0$. Since $\rho_0^{-k}r_k\equiv 1\bmod p$ and 
$p\equiv 0\bmod \Delta$, we have $\sum_{k=0}^{p-2}\rho_0^{-k}(r_0^k+\Delta_{r_0^k})^s\equiv-\delta_{s,1}\bmod \Delta$ for $1\leq s\leq i$.
Since $v_M(y/x)=pt+eb>0$, we have $x^{i-s}y^{s}\equiv 0\bmod x^{i-1}y$ for $1\leq s\leq i$. Therefore
$\Psi\cdot x^i\equiv ix^{i-1}\bmod x^{i-1}\Delta$. Since $\Psi\in \caH$, $\Psi\cdot L\subset L$. The result follows by evaluation $v_L(\Delta)$.
\end{proof}

The definition of $\caH$-scaffold in \cite[Definition
  2.3]{byott:A-scaffold} requires a shift parameter, which is the
integer $b_1=b$ defined in Theorem \ref{extensions}, two functions $\eub$
and $\eua$, which are $\eub:\{0,1,\ldots, p-1\}\rightarrow \bZ$
defined by $\eub(s)=sb$ and $\eua:\bZ\rightarrow \{0,1,\ldots, p-1\}$
defined by $\eua(t)\equiv -tb^{-1}\bmod p$. It requires elements
$\lambda_t=x^{\eua(t)}\pi_K^{f_t}\in L$ for $t\in \bZ$. Let $f_t$ be
defined by $t=-\eua(t)b+f_tp$. Therefore $v_L(\lambda_t)=t$ and
$\lambda_{t_1}\lambda_{t_2}^{-1}\in K$ when $t_1\equiv t_2\bmod p$, as
required. It requires an element $\Psi_1=\Psi\in \caH$ that because
Theorem \ref{scaffold} satisfies the required properties in order for there to 
be a $\caH$-scaffold of tolerance 
$$\euT=\begin{cases}\infty&\mbox{in }\chr(K)=p,\\
v_L(p)-(p-1)\ell&\mbox{in }\chr(K)=0,
\end{cases}$$
where, within $\bZ_{(p)}$, the integers localized at $p$,
$\ell\equiv b\bmod p$.

And so, similar to the discussion in \cite[\S4.1]{byott:A-scaffold}
we have
$0<\ell<v_L(p)/(p-1)$, and if
\begin{equation}\label{universal-scaff}
\ell <\frac{v_L(p)}{p-1}-2,
\end{equation} we can apply \cite[Theorem 3.1 and
  3.7]{byott:A-scaffold} to any ideal $\euP_L^n$ to
\begin{enumerate}
\item 
determine a basis for its associated order $$\euA_\caH(n)=\{h\in
\caH:h\euP_L^n\subseteq \euP_L^n\},$$ 
\item  determine that $\euA_\caH(n)$ is a local ring, with maximal ideal $\euM$ and residue field $\euA_\caH(n)/\euM\cong \kappa=\euO_L/\euP_L$,
\item  determine whether $\euP_L^n$
is free over $\euA_\caH(n)$. Indeed, 
\item  determine the number of
generators for $\euP_L^n$ over $\euA_\caH(n)$ if it is not free, and
\item  determine the embedding dimension $\dim_\kappa(\euM/\euM^2)$.
\end{enumerate}
Indeed, as a result of the scaffold, the results of 
\cite{ferton,aiba,de-smit:thomas,marklove, huynh}
under \eqref{universal-scaff}, proven for Galois extensions, hold
for non-Galois extensions as well.

In particular, as in the Introduction, if we set $0<\bar{b}<p$
such that $\bar{b}\equiv b\equiv \ell\bmod p$, then
\begin{enumerate}
\item For all $n\equiv \bar{b}\bmod p$, $\euP_L^{n}$ is free over its associated order $\euA_\caH(n)$.
\item For $n\equiv 0\bmod p$, $\euP_L^{n}$ is free over
  $\euA_\caH(n)$ if and only if $\bar{b}\mid p-1$. This includes $\euO_L$.
\item For $n\equiv \bar{b}+1\bmod p$, $\euP_L^{n}$ is free over
  $\euA_\caH(n)$ if and only if $\bar{b}=p-1$. This includes 
the inverse different $\caD_{L/K}^{-1}$.
\end{enumerate}
The first statement follows from 
\cite[Theorem 3.1]{byott:A-scaffold}.
The second and third statements follow from 
\cite[Corollary 3.3]{byott:A-scaffold}.

\section{Concluding remarks}
The definition of a scaffold, as presented in \cite{byott:A-scaffold},
was still evolving when the term, Galois scaffold, was coined in
\cite{elder:scaffold}. The intuition, as articulated in
\cite[\S1]{elder:scaffold}, was that extensions with Galois scaffolds
are somehow extensions that are no more complicated than ramified
cyclic extensions of degree $p$. A more mature intuition is now
available and is articulated in \cite[\S1]{byott:A-scaffold}. Still
the first intuition is useful, and now that scaffolds have been
defined more broadly than just for Galois extensions and classical
Galois module theory, the question arose whether scaffolds are
similarly present in ramified extensions of degree $p$ that are not
Galois and their Hopf-Galois structures. This paper considers
separable extensions. Elsewhere, evidence is provided for
inseparable extensions \cite[\S5]{byott:A-scaffold},
\cite[\S6]{elder:lambda-divided}.

\subsection*{Acknowledgements}
I would like to thank Alan Koch for discussions without which this
project would never have gotten started.  Thank Martha
Rzedowski-Calder\'{o}n for introducing me to the equation
\eqref{Sul-equation} and for clarifying its utility in the theory of
global function fields. Thank Lindsay Childs for his advice clarifying
the proof of Theorem \ref{hopf-gal-structure}. Thank Nigel Byott for
 important advice early on, but also, just as importantly, for our
many conversations over the years. Finally, I would like to thank the
University of Exeter for its generous hospitality in 2014-2015 while
this paper was being written.

\bibliographystyle{amsalpha}

\bibliography{bib}

\providecommand{\bysame}{\leavevmode\hbox to3em{\hrulefill}\thinspace}
\providecommand{\MR}{\relax\ifhmode\unskip\space\fi MR }
\providecommand{\MRhref}[2]{%
  \href{http://www.ams.org/mathscinet-getitem?mr=#1}{#2}
}
\providecommand{\href}[2]{#2}
\begin{thebibliography}{LRCMR12}

\bibitem[Aib03]{aiba}
Akira Aiba, \emph{Artin-{S}chreier extensions and {G}alois module structure},
  J. Number Theory \textbf{102} (2003), no.~1, 118--124. \MR{1994476
  (2004f:11127)}

\bibitem[Ama71]{amano}
Shigeru Amano, \emph{Eisenstein equations of degree {$p$} in a {${\eup}$}-adic
  field}, J. Fac. Sci. Univ. Tokyo Sect. IA Math. \textbf{18} (1971), 1--21.
  \MR{0308086 (46 \#7201)}

\bibitem[AT90]{artin:tate}
Emil Artin and John Tate, \emph{Class field theory}, second ed., Advanced Book
  Classics, Addison-Wesley Publishing Company, Advanced Book Program, Redwood
  City, CA, 1990. \MR{1043169 (91b:11129)}

\bibitem[BCE14]{byott:A-scaffold}
Nigel~P. Byott, Lindsay~N. Childs, and G.~Griffith Elder, \emph{Scaffolds and
  generalized integral {G}alois modules structure}, Preprint available at
  arXiv:1308.2088 [math.NT], April, 2014.

\bibitem[BEK]{elder:lambda-divided}
Nigel~P. Byott, G.~Griffith Elder, and Alan Koch, \emph{Action of the
  $\lambda$-divided power {H}opf algebra}, Preprint: April 28.

\bibitem[Chi89]{childs:separable}
Lindsay~N. Childs, \emph{On the {H}opf {G}alois theory for separable field
  extensions}, Comm. Algebra \textbf{17} (1989), no.~4, 809--825. \MR{990979
  (90g:12003)}

\bibitem[Chi00]{childs:book}
\bysame, \emph{Taming wild extensions: {H}opf algebras and local {G}alois
  module theory}, Mathematical Surveys and Monographs, vol.~80, American
  Mathematical Society, Providence, RI, 2000.

\bibitem[DF04]{dummit:foote}
David~S. Dummit and Richard~M. Foote, \emph{Abstract algebra}, third ed., John
  Wiley \& Sons, Inc., Hoboken, NJ, 2004. \MR{2286236 (2007h:00003)}

\bibitem[dST07]{de-smit:thomas}
Bart de~Smit and Lara Thomas, \emph{Local {G}alois module structure in positive
  characteristic and continued fractions}, Arch. Math. (Basel) \textbf{88}
  (2007), no.~3, 207--219. \MR{2305599 (2008b:11120)}

\bibitem[Eld09]{elder:scaffold}
G.~Griffith Elder, \emph{Galois scaffolding in one-dimensional elementary
  abelian extensions}, Proc. Amer. Math. Soc. \textbf{137} (2009), no.~4,
  1193--1203.

\bibitem[Fer73]{ferton}
Marie-Jos{\'e}e Ferton, \emph{Sur les id\'eaux d'une extension cyclique de
  degr\'e premier d'un corps local}, C. R. Acad. Sci. Paris S\'er. A-B
  \textbf{276} (1973), A1483--A1486. \MR{0332733 (48 \#11059)}

\bibitem[FV02]{fesenko:vostokov}
I.~B. Fesenko and S.~V. Vostokov, \emph{Local fields and their extensions},
  second ed., Translations of Mathematical Monographs, vol. 121, American
  Mathematical Society, Providence, RI, 2002, With a foreword by I. R.
  Shafarevich. \MR{1915966 (2003c:11150)}

\bibitem[GP87]{greither:pareigis}
Cornelius Greither and Bodo Pareigis, \emph{Hopf {G}alois theory for separable
  field extensions}, J. Algebra \textbf{106} (1987), no.~1, 239--258.

\bibitem[Hel91]{helou:1991}
Charles Helou, \emph{On the ramification breaks}, Comm. Algebra \textbf{19}
  (1991), no.~8, 2267--2279. \MR{1123123 (92g:11114)}

\bibitem[Huy14]{huynh}
Duc~Van Huynh, \emph{Artin-{S}chreier extensions and generalized associated
  orders}, J. Number Theory \textbf{136} (2014), 28--45. \MR{3145322}

\bibitem[JR06]{jones:roberts-paper}
John~W. Jones and David~P. Roberts, \emph{A database of local fields}, J.
  Symbolic Comput. \textbf{41} (2006), no.~1, 80--97. \MR{2194887
  (2006k:11230)}

\bibitem[LRCMR12]{martha:2012}
Florian Luca, Martha Rzedowski-Calder{\'o}n, and Myriam~Rosal{\'{\i}}a
  Maldonado-Ram{\'{\i}}rez, \emph{A generalization of a lemma of {S}ullivan},
  Comm. Algebra \textbf{40} (2012), no.~7, 2301--2308. \MR{2948828}

\bibitem[Lub13]{lubin:2013}
Jonathan Lubin, \emph{Elementary analytic methods in higher ramification
  theory}, J. Number Theory \textbf{133} (2013), no.~3, 983--999. \MR{2997782}

\bibitem[Mar13]{marklove}
Maria Marklove, \emph{Local {G}alois module structure in characteristic $p$},
  Ph.D. thesis, University of Exeter, 2013.

\bibitem[MW56]{mackenzie:whaples}
R.~E. MacKenzie and G.~Whaples, \emph{Artin-{S}chreier equations in
  characteristic zero}, Amer. J. Math. \textbf{78} (1956), 473--485.
  \MR{0090584 (19,834c)}

\bibitem[Ser79]{serre:local}
Jean-Pierre Serre, \emph{Local fields}, Graduate Texts in Mathematics, vol.~67,
  Springer-Verlag, New York-Berlin, 1979, Translated from the French by Marvin
  Jay Greenberg. \MR{554237 (82e:12016)}

\bibitem[Sul75]{sullivan}
Francis~J. Sullivan, \emph{{$p$}-torsion in the class group of curves with too
  many automorphisms}, Arch. Math. (Basel) \textbf{26} (1975), 253--261.
  \MR{0393035 (52 \#13846)}

\end{thebibliography}
\end{document}